\documentclass{amsart}
\usepackage{graphicx}

\newtheorem{theorem}{Theorem}[section]
\newtheorem{lemma}[theorem]{Lemma}

\newtheorem{proposition}[theorem]{Proposition}
\newtheorem{claim}[theorem]{Claim}

\theoremstyle{definition}

\theoremstyle{remark}
\newtheorem{remark}[theorem]{Remark}

\numberwithin{equation}{section}

\begin{document}

\title[Generalized torsion elements]{Generalized torsion elements in the knot groups of twist knots}


%    author one information
\author{Masakazu Teragaito}
\address{Department of Mathematics and Mathematics Education, Hiroshima University,
1-1-1 Kagamiyama, Higashi-hiroshima, Japan 739-8524.}
%\curraddr{}
\email{teragai@hiroshima-u.ac.jp}
\thanks{The author was partially supported by Japan Society for the Promotion of Science,
Grant-in-Aid for Scientific Research (C), 25400093.}

%    \subjclass is required.

\date{}

%\dedicatory{}

%    "Communicated by" -- provide editor's name; required.
\commby{}

%    Abstract is required.
\begin{abstract}
It is well known that any knot group is torsion-free, but it may admit
a generalized torsion element.
We show that the knot group of any negative twist knot admits
a generalized torsion element.
This is a generalization of the same claim for the knot $5_2$, which is
the $(-2)$-twist knot, by Naylor and Rolfsen.
\end{abstract}

\maketitle


%%%%%%%%%%%%%%%%%%%%%%%%%%%%%%%%%%%%%%%%%
\section{Introduction}

For a knot $K$ in the $3$-sphere $S^3$, the \textit{knot group\/} of $K$
is the fundamental group of the complement $S^3-K$.
It is a classical fact that any knot group is torsion-free.
However, many knot groups can have generalized torsion elements.
For a group, a non-trivial element is called a \textit{generalized torsion element\/}
if some non-empty finite product of its conjugates is the identity.
Typical exmples are 
the knot groups of torus knots
(see \cite{NR}).
%Similarly, this holds for any cable knot.
It was open whether the knot group of a hyperbolic knot
admits a generalized torsion element or not. 
However, Naylor and Rolfsen \cite{NR} first found such an element
in the knot group of the hyperbolic knot $5_2$.
As they wrote, the element was found with the help of computer, so
a (topological) meaning of the element is not obvious.

The knot $5_2$ is the $(-2)$-twist knot.
See below for the convention of twist knots.
In this article, we completely determine the existence
of generalized torsion elements in the knot groups of twist knots as a generalization of Naylor-Rolfsen's result. 
Figure \ref{fig:knot} shows the $m$-twist knot, where
the rectangle box contains the right-handed (resp.~left-handed)
horizontal $m$-full twists if $m>0$ (resp.~$m<0$).
By this convention, the $1$-twist knot is the figure-eight knot, and
the $(-1)$-twist knot is the right-handed trefoil as shown in Figure \ref{fig:knot}.
We may call the $m$-twist knot a positive or negative twist knot,
according to the sign of $m$.

\begin{figure}[tb]
\includegraphics*[scale=0.5]{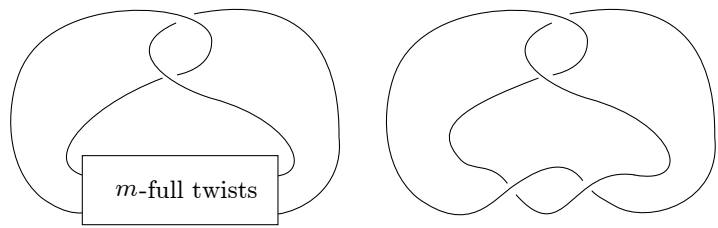}
\caption{The $m$-twist knot and the $(-1)$-twist knot}
\label{fig:knot}
\end{figure}

There is a well known relation between the bi-orderability and the existence
of generalized torsion elements in a group.
We recall that a group is said to be \textit{bi-orderable\/}
if it admits a strict total ordering,
which is invariant under multiplication on both sides.
Then any bi-orderable group has no generalized torsion elements (\cite{NR}).
The converse statement is not true, in general (\cite{MR}).

According to \cite{CDN}, the knot group of any positive twist knot
is bi-orderable, but
that of any negative twist knot is not bi-orderable.
Therefore, the former does not admit
a generalized torsion element, but
we can expect that the latter
would admit it.
In fact, Naylor and Rolfsen \cite{NR} confirmed this claim for the $(-2)$-twist knot.
The main result of the present article is to verify this expectation.

\begin{theorem}\label{thm:main}
The knot group of any negative twist knot admits
a generalized torsion element.
\end{theorem}

Throughout the paper, we use the convention 
$x^g=g^{-1}xg$ for the conjugate of $x$ by $g$ and 
$[x,y]=x^{-1}y^{-1}xy$ for the commutator.

%%%%%%%%%%%%%%%%%%%%%%%%%%%%%%%%%%%%%%%%%%%%%%%%%%%%%%%%
\section{Presentations of knot groups}\label{sec:pre}

For $n\ge 1$,
let $K$ be the $(-n)$-twist knot.
We prepare a certain presentation of the knot group of $K$ by using its
Seifert surface.
Figure \ref{fig:seifert} shows a Seifert surface $S$ where $n=2$.
The fundamental group $\pi_1(S)$ is a free group of rank two, generated by $x$ and $y$ as
illustrated there.
Since the complement of the regular neighborhood $N(S)$ of $S$ in $S^3$
is a genus two handlebody,
$\pi_1(S^3-N(S))$ is also a free group of rank two, generated by $a$ and $b$.

\begin{figure}[tb]
\includegraphics*[scale=0.5]{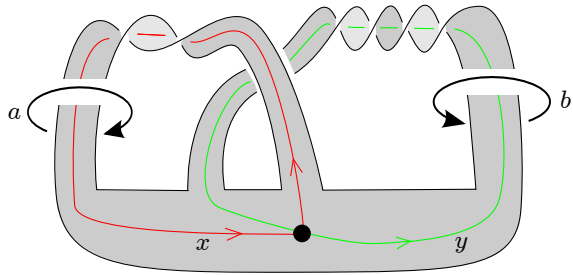}
\caption{A Seifert surface ($n=2$)}
\label{fig:seifert}
\end{figure}

\begin{theorem}\label{thm:pre}
The knot group $G$ of the $(-n)$-twist knot $K$ has a presentation
\[
\langle a,b,t\mid tat^{-1}=b^{-1}a, t(b^na^{-1})t^{-1}=b^n \rangle,
\]
where $t$ is a meridian.
\end{theorem}

\begin{proof}
If we push $x$ and $y$ off from the Seifert surface to the front side,
then we obtain $x^+=a$ and $y^+=b^{n}a^{-1}$.
Similarly, if we push them off to the back side,
then $x^-=b^{-1}a$ and $y^-=b^n$.
As shown in \cite[Lemma 2.1]{L} (it is just an application of the van Kampen theorem),
the knot group $G$ is generated by $a,b$ and a meridian $t$ with
two relations $tx^+t^{-1}=x^-$ and $ty^+t^{-1}=y^-$.
\end{proof}

From the second relation, we have
\begin{equation}\label{eq:a}
a=[t,b^{n}].
\end{equation}
Thus the knot group $G$ is generated by $b$ and $t$.

\begin{remark}
This presentation is distinct from one used in \cite{NR} when $n=2$, but
it is equivalent, of course.

By their choices, the generators $a$ and $b$ belong to the commutator subgroup of $G$.
Furthermore, we have $a=[t,b^n]$ and $b=[a^{-1},t^{-1}]$. 
This implies that these elements belong to the intersection of
all members of the lower central series of $G$.
\end{remark}

\begin{lemma}\label{lem:D}
Let $D=[t,b]$.  Then $D\ne 1$.
\end{lemma}

\begin{proof}
If $D=1$, then the knot group $G$ would be abelian, because $G$ is generated by $t$ and $b$.
Since the knot is non-trivial, this is impossible.
\end{proof}

\begin{lemma}\label{lem:a}
$a=DD^b \cdots D^{b^{n-1}}$.
\end{lemma}

\begin{proof}
In general, $[x,y^m]=[x,y][x,y^{m-1}]^y$.
The conclusion immediately follows from this and (\ref{eq:a}).
\end{proof}

%%%%%%%%%%%%%%%%%%%%%%%%%%%%%%%%%%%%%%%
\section{Proof of Theorem \ref{thm:main}}

In this section, we prove that $D=[t,b]$ is a generalized torsion element of $G$.
Let $\langle D\rangle$ be the semigroup consisting of all non-empty products
of conjugates of $D$ in $G$. 
That is,
\[
\langle D\rangle=\{D^{g_1}D^{g_2}\cdots D^{g_m} \mid g_1,g_2,\dots, g_m\in G, m\ge 1\}.
\]
If the semigroup $\langle D\rangle$ is shown to contain the identity, then
we can conclude that $D$ is a generalized torsion element.

From the second relation of the presentation in Theorem \ref{thm:pre},
$(tb^nt^{-1})(ta^{-1}t^{-1})=b^n$.
Then $tb^nt^{-1}=b^n(tat^{-1})$.
By the first relation of the presentation, we have
$tb^nt^{-1}=b^{n-1}a$.
Thus, by (\ref{eq:a}), we have
\begin{equation}\label{eq:base}
tb^nt^{-1}=b^{n-1}t^{-1}b^{-n}tb^n.
\end{equation}
This can also be modified to 
\begin{equation}\label{eq:base2}
b^{-(n-1)}tb^nt^{-1}=[t,b^n].
\end{equation}

\begin{proposition}\label{prop:n}
$b^n\in \langle D\rangle$.
\end{proposition}

\begin{proof}
We show $tb^nt^{-1}\in\langle D\rangle$, which
imples $b^n\in \langle D\rangle$.
By (\ref{eq:base}), 
it is sufficient to show 
$b^{n-1}t^{-1}b^{-n}tb^n\in\langle D\rangle$.

Conjugate $b^{n-1}t^{-1}b^{-n}tb^n$ by $t^{-(n-1)}$.
By applying (\ref{eq:base2}) repeatedly, we have
\begin{eqnarray*}
t^{n-1}(b^{n-1}t^{-1}b^{-n}tb^n)t^{-(n-1)}&=&
t^{n-1}b^{n-1}(t^{-1}b^{-1})\cdot b^{-(n-1)}tb^nt^{-1}\cdot t^{-(n-2)}\\
&=& t^{n-1}b^{n-1}(t^{-1}b^{-1})\cdot [t,b^n]\cdot t^{-(n-2)}\\
&=& t^{n-1}b^{n-1}(t^{-1}b^{-1})^2\cdot b^{-(n-1)}tb^nt^{-1} \cdot t^{-(n-3)}\\
&=& t^{n-1}b^{n-1}(t^{-1}b^{-1})^2\cdot [t,b^n]\cdot t^{-(n-3)}\\
& \vdots & \\
&=& t^{n-1}b^{n-1}(t^{-1}b^{-1})^{n-1}\cdot [t,b^n].
\end{eqnarray*}

Since $[t,b^n]=a\in \langle D\rangle$,
we can stop here when $n=1$.
Hereafter, we suppose $n\ge 2$.
Then, it suffices to show $t^{n-1}b^{n-1}(t^{-1}b^{-1})^{n-1}\in \langle D\rangle$.

\begin{claim}\label{cl:1}
$t^{n-1}b^{n-1}(t^{-1}b^{-1})^{n-1}=D^{(bt)^{n-2}b^{-(n-1)}t^{-(n-1)}}\cdot
t^{n-1}b^{n-1}(t^{-1}b^{-1})^{n-2}b^{-1}t^{-1}$.
\end{claim}

\begin{proof}[Proof of Claim \ref{cl:1}]
This follows from
\begin{eqnarray*}
D^{(bt)^{n-2}b^{-(n-1)}t^{-(n-1)}} &=&
t^{n-1}b^{n-1}(t^{-1}b^{-1})^{n-2}D(bt)^{n-2}b^{-(n-1)}t^{-(n-1)}\\
&=& t^{n-1}b^{n-1}(t^{-1}b^{-1})^{n-2}\cdot t^{-1}b^{-1}tb\cdot (bt)^{n-2}b^{-(n-1)}t^{-(n-1)}\\
&=& t^{n-1}b^{n-1}(t^{-1}b^{-1})^{n-1}\cdot tb(bt)^{n-2}b^{-(n-1)}t^{-(n-1)}.
\end{eqnarray*}
\end{proof}

If $n=2$, then the right hand side is 
$D^{b^{-1}t^{-1}}$, so we are done.

\begin{claim}\label{cl:2}
If $n\ge 3$, then
\[
t^{n-1}b^{n-1}(t^{-1}b^{-1})^{n-2}b^{-1}t^{-1}=
[t^{-1},b^{-(n-1)}]^{t^{-(n-2)}}\cdots [t^{-1},b^{-3}]^{t^{-2}}[t^{-1},b^{-2}]^{t^{-1}}.
\]
\end{claim}

\begin{proof}[Proof of Claim \ref{cl:2}]
We use an induction on $n$.
When $n=3$, the left hand side is $t^2b^2(t^{-1}b^{-1})b^{-1}t^{-1}$, which is
equal to $[t^{-1},b^{-2}]^{t^{-1}}$.

Assume that
\[
t^{n-2}b^{n-2}(t^{-1}b^{-1})^{n-3}b^{-1}t^{-1}=
[t^{-1},b^{-(n-2)}]^{t^{-(n-3)}}\cdots [t^{-1},b^{-3}]^{t^{-2}}[t^{-1},b^{-2}]^{t^{-1}}.
\]
Multiplying the left hand side $[t^{-1},b^{-(n-1)}]^{t^{-(n-2)}}$ from the left gives
\[
[t^{-1},b^{-(n-1)}]^{t^{-(n-2)}}t^{n-2}b^{n-2}(t^{-1}b^{-1})^{n-3}b^{-1}t^{-1}
=t^{n-1}b^{n-1}(t^{-1}b^{-1})^{n-2}b^{-1}t^{-1}.
\]
\end{proof}

Recall that $[t,b^m]\in \langle D\rangle$ for $m\ge 1$, as in the proof of Lemma \ref{lem:a}.
Also, $[t^{-1},b^{-m}]=[t,b^m]^{t^{-1}b^{-m}}\in \langle D\rangle$.
By Claims \ref{cl:1} and \ref{cl:2},
$t^{n-1}b^{n-1}(t^{-1}b^{-1})^{n-1}\in \langle D\rangle$.
This completes the proof.
\end{proof}

%%%%%%%%%%%%%%%%%%%%%%%%%%%%%%%%%%%%%%%%%%%
\begin{proposition}\label{prop:-n}
$b^{-n}\in \langle D\rangle$.
\end{proposition}

\begin{proof}
The argument is similar to that of the proof of Proposition \ref{prop:n}.
The relation (\ref{eq:base}) can be changed into
\begin{equation}\label{eq:3}
t^{-1}b^{-n}t=b^{-(n-1)}tb^nt^{-1}b^{-n}.
\end{equation}

We can see that the relation (\ref{eq:base}) is invariant
under the transformation $t\to t^{-1}$, $b\to b^{-1}$.
Thus, conjugating (\ref{eq:3}) by $t^{n-1}$
yields the relation
\[
t^{-n}b^{-n}t^n=t^{-(n-1)}b^{-(n-1)}(tb)^{n-1}\cdot [t^{-1},b^{-n}].
\]

When $n=1$, the right hand side is $[t^{-1},b^{-1}]=D^{b^{-1}t^{-1}}\in \langle D \rangle$.
So, assume $n\ge 2$.
As above, $[t^{-1},b^{-n}]\in \langle D \rangle$.

\begin{claim}\label{cl:3}
$t^{-(n-1)}b^{-(n-1)}(tb)^{n-1}=D^{(b^{-1}t^{-1})^{n-1}b^{n-1}t^{n-1}}\cdot
t^{-(n-1)}b^{-(n-1)}(tb)^{n-2}bt$.
\end{claim}

\begin{proof}[Proof of Claim \ref{cl:3}]
By the transformation $t\to t^{-1}$, $b\to b^{-1}$,
the relation of Claim \ref{cl:1} changes into
\[t^{-(n-1)}b^{-(n-1)}(tb)^{n-1}=[t^{-1},b^{-1}]^{(b^{-1}t^{-1})^{n-2}b^{n-1}t^{n-1}}\cdot
t^{-(n-1)}b^{-(n-1)}(tb)^{n-2}bt.
\]
Since $[t^{-1},b^{-1}]=D^{b^{-1}t^{-1}}$, we have the conclusion.
\end{proof}

Again, if $n=2$, then the right hand side is $D^{b^{-1}t^{-1}bt}=D$.

\begin{claim}
If $n\ge 3$, then
\[
t^{-(n-1)}b^{-(n-1)}(tb)^{n-2}bt=[t,b^{n-1}]^{t^{n-2}}\cdots [t,b^3]^{t^2}[t,b^2]^{t}.
\]
\end{claim}

\begin{proof}
This can be proved on an induction on $n$ as in the proof of Claim \ref{cl:2}.
We omit this.
\end{proof}

Since $[t,b^m]\in \langle D\rangle$ for $m\ge 1$,
we obtain $b^{-n}\in \langle D\rangle$.
\end{proof}

%%%%
\begin{proof}[Proof of Theorem \ref{thm:main}]
By Propositions \ref{prop:n} and \ref{prop:-n},
both of $b^n$ and $b^{-n}$ belong to the semigroup $\langle D\rangle$.
Hence $\langle D\rangle$ contains the identity.
Since $D\ne 1$ by Lemma \ref{lem:D},
$D$ is a generalized torsion element.
\end{proof}

\begin{remark}
For the case $n=2$, Naylor and Rolfsen \cite{NR} showed that
$[b^{-1},t]$ in our notation is a generalized torsion element.
This equals to $D^{b^{-1}}$, which is a conjugate of our $D$.
\end{remark}

%%%%%%%%%%%%%%%%%%%%%%%%%%
\section{Comments}

A group without generalized torsion elements is called an $R^*$-group in literature (\cite{MR}).
It was an open question whether the class of bi-orderable groups
coincided with the class of $R^*$-groups \cite[p.79]{MR}.
Unfortunately, the answer is known to be negative.
However, there might be a possibility that two classes coincide among knot groups.
Thus we expect that any non-bi-orderable knot group
would admit a generalized torsion element.
Many knot groups are now known to be non-bi-orderable by \cite{CGW,CDN,CR,PR}.
It would be a next problem to examine such knot groups.

\section*{Acknowledgments}
The author would like to thank Makoto Sakuma for helpful conversations.

%The author would like to thank the referee for valuable suggestions and comments, and
%Mario Eudave-Mu\~{n}oz for pointing out an error in the original manuscript.

%%%%%%%%%%%%%%%%%%%%%%%%%%%%%%%%%%%%%%%%%%%%%%%%%%%%%%%%%%%%%%%%%%%%%%
\bibliographystyle{amsplain}

\end{document}